\documentclass[a4paper,11pt]{article}
\usepackage{amsmath,amssymb,enumerate,color}
\usepackage{amsthm}
\usepackage{txfonts}

\newcommand{\R}{\mathbb{R}}

\newcommand{\N}{\mathbb{N}}
\newcommand{\ep}{\varepsilon}
\newcommand{\pa}{\partial}

\DeclareMathOperator{\supp}{supp}
\DeclareMathOperator{\lifespan}{LifeSpan}
\newcommand{\lr}[1]{{}\langle{}#1{}\rangle{}}

\newtheorem{theorem}{Theorem}[section]
\newtheorem{lemma}[theorem]{Lemma}
\newtheorem{proposition}[theorem]{Proposition}

\theoremstyle{remark}
\newtheorem{remark}{Remark}[section]
\theoremstyle{definition}

\newtheorem{definition}{Definition}[section]

\numberwithin{equation}{section}

\makeatletter
\def\@cite#1#2{[{{\bfseries #1}\if@tempswa , #2\fi}]}
\makeatother                  
\begin{document}
\begin{center}
\Large{{\bf
Life-span of blowup solutions to semilinear wave equation 
\\with space-dependent critical damping 
}}
\end{center}

\vspace{5pt}

\begin{center}
Masahiro Ikeda%
\footnote{
Department of Mathematics, Faculty of Science and Technology, Keio University, 3-14-1 Hiyoshi, Kohoku-ku, Yokohama, 223-8522, Japan/Center for Advanced Intelligence Project, RIKEN, Japan, 
E-mail:\ {\tt masahiro.ikeda@keio.jp/masahiro.ikeda@riken.jp}}
and
Motohiro Sobajima%
\footnote{
Department of Mathematics, 
Faculty of Science and Technology, Tokyo University of Science,  
2641 Yamazaki, Noda-shi, Chiba, 278-8510, Japan,  
E-mail:\ {\tt msobajima1984@gmail.com}}
\end{center}

\newenvironment{summary}{\vspace{.5\baselineskip}\begin{list}{}{%
     \setlength{\baselineskip}{0.85\baselineskip}
     \setlength{\topsep}{0pt}
     \setlength{\leftmargin}{12mm}
     \setlength{\rightmargin}{12mm}
     \setlength{\listparindent}{0mm}
     \setlength{\itemindent}{\listparindent}
     \setlength{\parsep}{0pt}
     \item\relax}}{\end{list}\vspace{.5\baselineskip}}
\begin{summary}
{\footnotesize {\bf Abstract.}
This paper is concerned with the blowup phenomena for 
initial value problem of semilinear wave equation 
with critical space-dependent damping term
\begin{equation}\label{DW}
\tag*{(DW:$V_0$)}
\begin{cases}
\pa_t^2 u(x,t) -\Delta u(x,t) + V_0|x|^{-1}\pa_t u(x,t)=|u(x,t)|^p, 
& (x,t)\in \R^N \times (0,T),
\\
u(x,0)=\ep f(x), 
& x\in \R^N,
\\ 
\pa_t u(x,0)=\ep g(x),  
& x\in \R^N,
\end{cases}
\end{equation}
where $N\geq 3$, $V_0\in [0,\frac{(N-1)^2}{N+1})$, 
$f$ and $g$ are compactly supported smooth functions 
and $\ep>0$ is a small parameter. The main result of the present paper 
is to give a solution of \ref{DW} and 
to provide a sharp estimate for lifespan for such a solution when 
$\frac{N}{N-1}<p\leq p_S(N+V_0)$, where $p_S(N)$ 
is the Strauss exponent for (DW:$0$). 
The main idea of the proof is due to the technique of test functions 
for (DW:$0$) originated by Zhou--Han (2014, MR3169791). 
Moreover, we find a new threshold value $V_0=\frac{(N-1)^2}{N+1}$ 
for the coefficient of critical and singular damping $|x|^{-1}$.
}
\end{summary}



\section{Introduction}
In this paper we consider the blowup phenomena for 
initial value problem of semilinear 
wave equation with scale-invariant damping term of space-dependent type as follows:
\begin{equation}\label{ndw}
\begin{cases}
\pa_t^2 u(x,t) -\Delta u(x,t) + a(x)\pa_t u(x,t)=|u(x,t)|^p, 
& (x,t)\in \R^N \times (0,T),
\\
u(x,0)=\ep f(x), 
& x\in \R^N,
\\ 
\pa_t u(x,0)=\ep g(x),  
& x\in \R^N, 
\end{cases}
\end{equation}
where $N\geq 3$, $a(x)=V_0|x|^{-1}$ $(V_0\geq 0)$, 
$\ep>0$ is a small parameter
and 
$f,g$ 
are smooth nonnegative functions satisfying $g\not\equiv 0$ with 
\[
{\rm supp}(f,g)\subset \overline{B(0,R_0)}=\{x\in\R^N\;;\;|x|\leq R_0\}
\]
for some $R_0>0$. 
Note that by taking $u_\lambda(x,t)=\lambda^{\frac{2}{p-1}}u(\lambda x,\lambda t)$ with $\lambda=R_0$, 
we can always assume $R_0=1$ without loss of generality.

The study of blowup phenomena for \eqref{ndw} with $N=3$ and $V_0=0$ 
was initially started by F. John in \cite{John79} for $1<p<1+\sqrt{2}$. 
Strauss conjectured in \cite{Strauss81} that 
the number $p_0(N)$ given by the positive root of 
the quadratic equation
\[
(N-1)p^2-(N+1)p-2=0
\]
is the threshold for dividing 
the following two situations:
blowup phenomena at a finite time
for arbitrary small initial data 
and global existence of small solutions.
The conjecture of Strauss was completely 
solved until Yordanov--Zhang \cite{YZ06} and 
Zhou \cite{Zhou07}. 

After that the lifespan of solutions to 
nonlinear wave equations (\eqref{ndw} with $V_0=0$)
with small initial data
has been considered by many authors. 
If $1<p<p_0(N)$, then 
by 
Sideris \cite{Sideris84} 
and Di Pomponio--Georgiev \cite{DG01}
we have the two-sided estimates 
for lifespan of solution with small initial data as 
\[
c\ep^{\frac{2p(p-1)}{2+(N+1)p-(N-1)p^2}+\delta}
\leq 
\lifespan(u)
\leq 
C\ep^{\frac{2p(p-1)}{2+(N+1)p-(N-1)p^2}}
\]
with arbitrary small $\delta>0$. 
For the critical case $p=p_0(N)$, 
Takamura--Wakasa \cite{TW11} succeeded 
in proving sharp upper bound of 
lifespan 
\[
\exp[c\ep^{-p(p-1)}]
\leq 
\lifespan(u)
\leq 
\exp[C\ep^{-p(p-1)}]
\]
for remaining case $N=4$, and by 
\cite{TW11} the study of 
the lifespan for blowup solutions 
to nonlinear wave equations with small data 
has been completed (for the other contributions see e.g, \cite{TW11} and its references therein). 
In the connection to the previous paper,
we have to remark that Zhou--Han \cite{ZH14} 
gave a short proof for verifying the sharp upper 
bound of lifespan by using an estimate 
established in \cite{YZ06} 
and a kind of test functions 
including the Gauss hypergeometric functions. 

In this paper, we mainly deal with the problem 
\eqref{ndw} with $N\geq 3$ and $V_0>0$. 
Because of the strong singularity of damping term at the origin, 
the study of \eqref{ndw} has not been considered so far. 
Since the problem has a scaling-invariant structure, 
one can expect that some threshold for $V_0$ appears.

The first purpose of this paper is 
to clarify the local wellposedness of \eqref{ndw} for 
$1<p<\frac{N-2}{N-4}$ in solutions in $H^2(\R^N)$. 
The second is to show an upper bound of the lifespan 
of solutions to \eqref{ndw} with respect to small parameter 
$\ep>0$ 
and to pose a threshold number for $V_0$ 
dividing completely different situations. 

The first assertion of this paper is for 
local wellposedness of \eqref{ndw}. 

\begin{proposition}\label{prop:solve}
Let $N\geq 3$, $V_0\geq  0$ and 
\begin{align*}
\begin{cases}
1<p<\infty &\text{if}\ N=3,4
\\
1<p<\frac{N-2}{N-4} &\text{if}\ N\geq 5.
\end{cases}
\end{align*}
For every $(f,g)\in H^2(\R^N)\cap H^1(\R^N)$ and $\ep>0$, 
there exist $T=T(\|f\|_{H^2},\|g\|_{H^1},\ep)>0$ and 
a unique strong solution of \eqref{ndw} in the following class:
\[
u\in S_T=C^2([0,T];L^2(\R^N))\cap C^1([0,T];H^1(\R^N))\cap C([0,T];H^2(\R^N))
\]
Moreover, one has for every $t\geq 0$, 
\[
\supp u(t)\subset \overline{B(0,R_0+t)}.
\]
\end{proposition}

\begin{definition}\label{def:lifespan}
We denote $\lifespan(u)$ as the maximal existence time for solution of \eqref{ndw}, that is, 
\[
\lifespan(u)=\sup\{T>0\;;\;u\in S_T\ \& \ \text{$u$ is a solution of \eqref{ndw} in $(0,T)$}\}
\]
\end{definition}

\begin{definition}\label{def:strauss.exp}
We introduce the following quadratic polynomial
\[
\gamma(n;p)=2+(n+1)p-(n-1)p^2
\]
and denote $p_0(n)$ as the positive root of 
the quadratic equation $\gamma(n;p)=0$ 
as in Introduction. We also put 
\[
V_*=\frac{(N-1)^2}{N+1}
\]
and for areas for $(p,V_0)$ as follows:
\begin{gather}
\Omega_0=
\left\{(p,V_0)
\;;\;
p=p_0(N+V_0), \quad 0\leq 
V_0< V_*
\right\}\\
\Omega_1=
\left\{(p,V_0)
\;;\;
\max\left\{p_0(N+2+V_0),\frac{2}{N-1-V_0}\right\}\leq p < p_0(N+V_0), 
\quad 
0\leq V_0< V_*
\right\}
\\
\Omega_2=
\left\{(p,V_0)
\;;\;
\frac{2(N+1)}{N+1+V_0}<p<\frac{2}{N-1-V_0}, 
\quad 
\frac{(N+1)(N-2)}{N+2}<V_0<V_*
\right\}
\\
\Omega_3=
\left\{(p,V_0)
\;;\;
\max\left\{\frac{N}{N-1},\frac{N+3+V_0}{N+1+V_0}\right\}
<p<
\max\left\{p_0(N+2+V_0),\frac{2(N+1)}{N+1+V_0}\right\}
\right\}
\end{gather}
\end{definition}

\bigskip

\input{blowup.fig}

\bigskip
\bigskip

Now we are in a position to state 
our main result in this paper about 
upper bound of lifespan of solutions to \eqref{ndw}.

\begin{theorem}\label{main}
Let $\frac{N}{N-1}<p<\infty $ if $N=3,4$ 
and $\frac{N}{N-1}<p<\frac{N-2}{N-4}$ if $N\geq 5$. 
Fix $(f,g)$ satisfying $f\geq 0$, $g\geq 0$, $g\not\equiv 0$ and 
${\rm supp}(f,g)\subset \overline{B(0,1)}$.
Let $u_\ep$ be the solution of \eqref{ndw} in Proposition \ref{prop:solve} with the parameter $\ep>0$. 
If $(p,V_0)\in \Omega_0\cup \Omega_1 \cup \Omega_2\cup \Omega_3$, then  
$\lifespan(u_\ep)<\infty$. 
Moreover, one has
\begin{gather}
\lifespan(u_\ep)
\leq 
\begin{cases}
\exp[C\ep^{-p(p-1)}]
& \text{if}\ (p,V_0)\in \Omega_0,
\\[10pt]
C_\delta' \ep^{-2p(p-1)/\gamma(N+V_0;p)-\delta} 
& \text{if}\ (p,V_0)\in \Omega_1,
\\[10pt]
C_\delta'' \ep^{-\frac{2(p-1)}{2N-(N-1+V)p}-\delta} 
& \text{if}\ (p,V_0)\in \Omega_2,
\\[10pt]
C_\delta''' \ep^{-1-\delta} 
& \text{if}\ (p,V_0)\in \Omega_3,
\end{cases}
\end{gather}
where $\delta$ can be chosen arbitrary small
and $C_\delta$, $C'_\delta$, $C''_\delta$ and $C'''_\delta$ 
are positive constants which depend on all parameters
without $\ep$.
\end{theorem}

\begin{remark}
We emphasize the following two facts. 
The proof of \cite{ZH14} 
depends on an estimate established by \cite{YZ06} 
(for detail, see \cite[(2,5')]{YZ06}), 
however, our proof does not depend on that. 
The proof of Theorem \ref{main} can be 
applicable to weaker solutions of \eqref{ndw} 
belonging to 
$C([0,T));H^1(\R^N)\cap L^p((0,T)\times \R^N)$. 
\end{remark}

\begin{remark}
Taking the threshold value 
$V_0=V_*$ 
formally, we have
\[
\gamma\left(N+V_*;p\right)
=
2\left(1+Np\right)
\left(1-\frac{N-1}{N+1}p\right)
\]
and therefore $p_0(N+V_*)=\frac{N+1}{N-1}=1+\frac{2}{N-1}$. 
On the one hand, 
critical exponent for the blowup phenomena for the problem 
\begin{equation}\label{ndw.alpha}
\begin{cases}
\pa_t^2 u(x,t) -\Delta u(x,t) + \lr{x}^{-\alpha}\pa_t u(x,t)=|u(x,t)|^p, & (x,t)\in \R^N \times (0,T),
\\
u(x,0)=\ep f(x), & x\in \Omega,
\\ 
\pa_t u(x,0)=\ep g(x),  & x\in \Omega, 
\end{cases}
\end{equation}
is given by $p_F(\alpha)=1+\frac{2}{N-\alpha}$, $\alpha\in [0,1)$ 
which is so-called Fujita exponent 
(see e.g., 
Ikehata--Todorova--Yordanov \cite{ITY09}
and also Ikeda--Ogawa \cite{IO16}). 
We formally put again a threshold value $\alpha=1$. 
Then one can find 
\[
p_0\left(N+V_0\right)=p_F(1). 
\]
The left-hand side comes from the blowup phenomena 
for nonlinear wave equation and the right-hand side comes from 
the one for nonlinear heat equation. 
In this connection, we would conjecture that 
if $V_0>V_*$, then the threshold of 
blowup phenomena is given by the Fujita exponent $p_F(1)$. 
\end{remark}

\begin{remark}
If $(p,V_0)\in \Omega_0 \cup \Omega_1$, 
then Theorem \ref{main} seems to give a sharp lifespan 
of solutions to \eqref{ndw} with small initial data. 
In the case $(p,V_0)\in \Omega_3$, we cannot derive 
the estimates for lifespan with $\ep^{-\tau}$ with $\tau$ less than one. 
So the estimate in $\Omega_3$ seems not to be sharp. 
For the case $(p,V_0)\in \Omega_2$ 
the effect of diffusion structure seems to appear in the estimate. 
\end{remark}

The present paper is organized as follows. 
In Section 2, we first give existence and uniqueness 
of local-in-time solutions to \eqref{ndw} if $p\leq \frac{N-2}{N-4}$
by using the standard semigroup properties.
In Section 3, we construct special solutions 
of linear wave equation with anti-damping term $-\frac{V_0}{|x|}\pa_tu$.
In this point we use the idea due to \cite{ZH14} 
(they only considered the case $V_0=0$), 
which will be a test function for proving blowup phenomena. 
In Section 4, we prove blowup phenomena by dividing 
two cases $p<p_0(N+V_0)$ and $p=p_0(N+V_0)$.

\section{Local solvability of nonlinear wave equation with singular damping}

In this section we construct a solution of \eqref{ndw} 
with initial data belonging to $H^2(\R^N)\times H^1(\R^N)$. 
To do so, we first treat the linear problem
\begin{equation}\label{dw0}
\begin{cases}
\pa_t^2 u(x,t) -\Delta u(x,t) + a(x)\pa_t u(x,t)=0, 
& (x,t)\in \R^N \times (0,\infty),
\\
u(x,0)=u_0(x), \quad \pa_tu(x,0)=u_1(x),& x\in \R^N. 
\end{cases}
\end{equation}

\subsection{$C_0$-Semigroup for linear wave equation with singular damping}
Now we start with the usual $N$-dimensional Laplacian
\[
Au:=-\Delta u, \quad D(A)=H^2(\R^N). 
\]
We note that $A$ is $m$-accretive in $L^2(\R^N)$, that is, 
$(I+A)D(A)=L^2(\R^N)$ and 
$\left(-\Delta u,u\right)\geq 0$ for every $u\in H^2(\R^N)$.
Set $\mathcal{H}=H^1(\R^N)\times L^2(\R^N)$ and 
\[
\mathcal{A}
=
\begin{pmatrix}
0 & -1 \\
A & 0 
\end{pmatrix}, 
\quad 
D(\mathcal{A})
=
H^2(\R^N)\times H^1(\R^N)
\]
and 
\[
\mathcal{B}
=
\begin{pmatrix}
0 & 0 \\
0 & |x|^{-1} 
\end{pmatrix}, 
\quad 
D(\mathcal{B})
=
H^1(\R^N)\times 
\{v\in L^2(\R^N)\;;\;|x|^{-1}v\in L^2(\R^N)\}
\]
and put 
\[
\mathcal{A}_\kappa=\mathcal{A}+\kappa\mathcal{B}, \quad 
D(\mathcal{A}_\kappa)=D(\mathcal{A})\cap D(\mathcal{B})
\]
Then in view of Hille--Yosida theorem, 
we have the following $C_0$-semigroup on $\mathcal{H}$ (see e.g., Pazy 
\cite{Pazybook}).
\begin{lemma}\label{lem:linear}
Let $N\geq 3$. For every $V\geq 0$, $\frac{1}{2}I+\mathcal{A}_\kappa$ is $m$-accretive 
in $\mathcal{H}$.
Therefore $-\mathcal{A}_{\kappa}$ generates a $C_0$-semigroup 
$\{T_{\kappa}(t)\}_{t\geq0}$ on $\mathcal{H}$. 
Moreover, if $\supp (u_0,u_1)\subset \overline{B(0,R)}$, then 
$\supp [T_{\kappa}(t)(u_0,u_1)]\subset \overline{B(0,R+t)}$.
\end{lemma}
\begin{proof}
By Hardy's inequality we have $D(\mathcal{A})\subset D(\mathcal{B})$. This means that 
$D(\mathcal{A}_\kappa)=D(\mathcal{A})\cap D(\mathcal{B})=D(\mathcal{A}).$

\smallskip

\noindent{\bf (Accretivity)} By interation by parts, we have
\[
\left(
\mathcal{A}
\begin{pmatrix}
u\\v
\end{pmatrix},\begin{pmatrix}
u\\v
\end{pmatrix}
\right)_{\mathcal{H}}
=
\left(
\begin{pmatrix}
-v\\-\Delta u
\end{pmatrix},\begin{pmatrix}
u\\v
\end{pmatrix}
\right)_{\mathcal{H}}
=
\int_{\R^N}
\Big(
  -vu-\nabla v\cdot\nabla u -(\Delta u)v
\Big)\,dx
=
-\int_{\R^N}
  uv
  \,dx.
\]
Since $\kappa\mathcal{B}$ is clearly accretive, we have the accretivity of $\mathcal{A}_\kappa$. 

\smallskip

\noindent{\bf (Maximality)}
Let $F=(f,g)\in \mathcal{H}$. Then $\lambda U+\mathcal{A}U+\kappa\mathcal{B}U=F$ 
is equivalent to the system
\[
\lambda u -v=f, \quad 
\lambda v -\Delta u +\kappa|x|^{-1}v=g.
\]
Substituting $v=\lambda u - f$, we see that
\[
\lambda^2 u -\Delta u +\lambda \kappa |x|^{-1} u =g+\lambda f+k|x|^{-1}f=f_{\lambda,\kappa}.
\]
Taking $\widetilde{u}(y)=u(\lambda^{-1} y)$ and $\widetilde{f_{\lambda,\kappa}}(y)=f_{\lambda,\kappa}(\lambda^{-1} y)$ yields that
\[
\widetilde{u} -\Delta \widetilde{u} +k|y|^{-1} \widetilde{u} 
=
\lambda^{-2}
\widetilde{f_{\lambda,\kappa}}. 
\]
This is nothing but the resolvent problem of the Schr\"odinger operator with 
positive Coulomb potentials. Therefore there exists 
$\widetilde{u}_{\lambda,k}\in H^2(\R^N)$ such that 
\[
\widetilde{u}_{\lambda,k} -\Delta \widetilde{u}_{\lambda,k} +k|y|^{-1} \widetilde{u}_{\lambda,k}
=
\lambda^{-2}
\widetilde{f_{\lambda,k}}. 
\]
Putting $u_{\lambda,k}(x)=\widetilde{u}_{\lambda,k}(\lambda^{2} x)\in H^2(\R^N)$, we obtain 
\[
\lambda^2 u_{\lambda,k} -\Delta u_{\lambda,k} +\lambda k|x|^{-1} u_{\lambda,k} =f_{\lambda,k}=g+\lambda f+k|x|^{-1}f.
\]
Finally setting 
$v_{\lambda,k}=\lambda u_{\lambda,k} -f\in H^1(\R^N)$, we obtain 
($\lambda I+\mathcal{A}+k\mathcal{B})(u_{\lambda,k},v_{\lambda,k})=(f,g)$. 

The finite propagation property follows from the standard argument 
for wave equation with regular damping term. The proof is complete.
\end{proof}

\subsection{Local solvability of nonlinear problem}

We consider \eqref{ndw} with initial data $(u(0),\pa_tu(0))=U_0=(u_0,u_1)\in H^2(\R^N)\times H^1(\R^N)$ 
which is equivalent to the following problem
\[
\pa_t
\begin{pmatrix}
u(t)\\v(t)
\end{pmatrix}
+\mathcal{A}_\kappa
\begin{pmatrix}
u(t)\\v(t)
\end{pmatrix}
=
\begin{pmatrix}
0\\
\mathcal{N}(u(t),v(t))
\end{pmatrix}
\]
with $\mathcal{N}(u,v)=(0,|u|^{p})$. 
Here we construct the 
corresponding  mild solution in $H^2(\R^N)\times H^1(\R^N)$ given by 
\[
\begin{pmatrix}
u(t)\\v(t)
\end{pmatrix}
=
T_\kappa(t)
\begin{pmatrix}
u_0\\u_1
\end{pmatrix}
+
\int_{0}^t
T_\kappa(t-s)
\begin{pmatrix}
0\\
\mathcal{N}(u(s),v(s))
\end{pmatrix}
\,ds,
\]
where $\{T_\kappa(t)\}_{t\geq 0}$ is determined in Lemma \ref{lem:linear}.

\begin{lemma}\label{lem:XT}
{\bf (i)} 
The following metric space
\[
X_T=\Big\{
U\in 
C([0,T];\mathcal{H})
\cap L^\infty([0,T];D(\mathcal{A}_\kappa))
\;;\;
\sup_{0<t<T}\|U(t)\|_{D(\mathcal{A}_\kappa)}\leq M
\Big\}, 
\quad 
M:=2(\|U_0\|_{D(\mathcal{A}_\kappa)}+1)
\]
with the distance
\[
d(U_1,U_2):=
\max_{0\leq t\leq T}
\big\|U(t)-U_2(t)\big\|_{\mathcal{H}}, \quad U_1,U_2\in X_T.
\]
is complete. 

\medskip 

\noindent
{\bf (ii)} 
If $\supp (u_0,u_1)\subset \overline{B(0,R)}$, 
then the following metric space
\[
Y_T=\Big\{
U\in 
C([0,T];\mathcal{H})
\cap L^\infty([0,T];D(\mathcal{A}_\kappa))
\;;\;
\sup_{0<t<T}\|U(t)\|_{D(\mathcal{A}_\kappa)}\leq M, 
\quad 
\supp (u(t),v(t))\subset \overline{B(0,R+t)}
\Big\}
\]
with the same distance $d$ is also complete. 
\end{lemma}

\begin{proof}
Take a Cauchy sequence $\{U_n\}_{n\in\N}$ in $X_T$. 
The completeness of $C([0,T];\mathcal{H})$ yields that there exists 
$U_\infty\in C([0,T];\mathcal{H})$ such that 
\[
U_n\to U_\infty\quad \text{strongly in }C([0,T];\mathcal{H})\quad \text{ as }n\to \infty.
\]
Moreover, we can subtract a subsequence $U_{n_j}$ and $\widetilde{U}\in L^\infty([0,T];D(\mathcal{A}_\kappa))$
such that 
\[
\sup_{0<t<T}\|\widetilde{U}(t)\|_{D(\mathcal{A}_\kappa)}\leq M\]
 and 
\[
U_{n_j}\to \widetilde{U}\quad \text{$*$-weakly in } L^\infty([0,T];D(\mathcal{A}_\kappa))\quad \text{ as }j\to \infty.
\]
Therefore we have $U_\infty=\widetilde{U}$. Therefore $X_T$ is a complete metric space.
If $\supp U_n(t)\subset \{x;|x|\leq R+t\}$, then 
by strong convergence we have $\supp U_\infty(t)\subset \{x;|x|\leq R+t\}$. 
This means that $Y_T$ is also complete. 
\end{proof}

\begin{lemma}\label{lem:mild}
There exists $T_0$ such that 
$\Psi:X_{T_0}\to X_{T_0}$ and $\Psi:Y_{T_0}\to Y_{T_0}$ 
are both well-defined, and $\Psi$ 
is contractive in $X_{T_0}$ and in $Y_{T_0}$. 
\end{lemma}
\begin{proof}
First observe that by finite propagation property in Lemma \ref{lem:linear}, 
we can deduce $\supp \Psi U(t)\subset \overline{B(0,R+t)}$ 
when $\supp U_0 \subset \overline{B(0,R)}$.  
Since $X_{T}$ and $Y_T$ are endowed with the same distance, 
It suffices to prove the assertion for $X_T$. 

We recall that the norms in $D(\mathcal{A}_\kappa)$ and 
in $H^2(\R^N)\cap H^1(\R^N)$ are equivalent. 

\noindent
{\bf (well-defined)} If $U=(u,v)\in X_T$, then
\begin{align*}
\|\mathcal{N}(U(t))\|_{D(\mathcal{A}_\kappa)}^2
&\leq 
C_\kappa^2\|\mathcal{N}(U(t))\|_{H^2\times H^1}^2
\\
&=
C_\kappa^2\big\||u(t)|^p\big\|_{H^1}^2
\\
&=
C_\kappa^2
\int_{\R^N}
\Big(|u(t)|^p+|\nabla(|u(t)|^p)|^2\Big)\,dx
\\
&=
C_\kappa^2
\left(\int_{\R^N}
|u(t)|^{2p}
\,dx
+p^2
\int_{\R^N}|u(t)|^{2(p-1)}|\nabla u(t)|^2\,dx
\right)
\\
&\leq
C_\kappa^2
\left(\int_{\R^N}|u(t)|^{N(p-1)}\,dx\right)^{\frac{2}{N}}
\left(\int_{\R^N}|u(t)|^{\frac{2N}{N-2}}\,dx\right)^{1-\frac{2}{N}}
\\
&\quad +p^2C_\kappa^2
\left(\int_{\R^N}|u(t)|^{N(p-1)}\,dx\right)^{\frac{2}{N}}
\left(\int_{\R^N}|\nabla u(t)|^{\frac{2N}{N-2}}\,dx\right)^{1-\frac{2}{N}}.
\end{align*}
Since $p\leq \frac{N-2}{N-4}$, we have $N(p-1)\leq (\frac{1}{2}-\frac{2}{N})$ and therefore
\begin{align*}
\|\mathcal{N}(U(t))\|_{D(\mathcal{A}_\kappa)}
&\leq C_\kappa'\|u(t)\|_{H^2}
\\
&\leq C_\kappa''\|U(t)\|_{H^2\times H^1}
\\
&\leq C_\kappa''\|U(t)\|_{D(\mathcal{A}_\kappa)}
\\
&\leq C_\kappa''M.
\end{align*}
Therefore we have for $0\leq t\leq T\leq \log 2$, 
\begin{align*}
\|\Psi U(t)\|_{D(\mathcal{A}_\kappa)}
&\leq 
\|T_\kappa(t)U_0\|_{D(\mathcal{A}_\kappa)}
+
\int_{0}^t
  \|T_\kappa(t-s)[\mathcal{N}(U(s))]\|_{D(\mathcal{A}_\kappa)}
\,ds  
\\
&\leq 
e^{t/2}\|U_0\|_{D(\mathcal{A}_\kappa)}
+
\int_{0}^t
  e^{(t-s)/2}\|\mathcal{N}(U(s))\|_{D(\mathcal{A}_\kappa)}
\,ds  
\\
&\leq 
\sqrt{2}\|U_0\|_{D(\mathcal{A}_\kappa)}
+
\sqrt{2}
  C_\kappa''Mt.
\end{align*}
Therefore there exists $T_1\in (0,\log 2]$ 
such that $\sup_{0<t<T_1}\|\Psi U(t)\|_{D(\mathcal{A}_\kappa)}\leq M$.

Next we prove continuity of $\Psi U$ on $[0,T]$.
For  $0\leq t_1< t_2\leq T$, 
\begin{align*}
\|\Psi U(t_1)-\Psi U(t_2)\|_{\mathcal{H}}
&\leq 
\|[T_\kappa(t_1)-T_\kappa(t_2)]U_0\|_{\mathcal{H}}
\\
&\quad
+
\left\|
\int_0^{t_1}T(t_1-s)N(U(s))\,ds
-
\int_0^{t_2}T(t_2-s)N(U(s'))\,ds'
\right\|_{\mathcal{H}}
\\
&\leq 
|t_1-t_2|\,\|\mathcal{A}_{k}U_0\|_{\mathcal{H}}
\\
&\quad
+
\left\|
\int_0^{t_1}[I-T(t_2-t_1)]T(t_1-s)N(U(s))\,ds
\right\|_{\mathcal{H}}
+
\left\|
\int_{t_1}^{t_2}T(t_2-s)N(U(s'))\,ds'
\right\|_{\mathcal{H}}
\\
&\leq 
|t_1-t_2|\,\|\mathcal{A}_{k}U_0\|_{\mathcal{H}}
\\
&\quad
+
e^{T/2}|t_2-t_1|
\int_0^{t_1}\|\mathcal{A}_\kappa N(U(s))\|_{\mathcal{H}}
\,ds
+
e^{T/2}
\int_{t_1}^{t_2}\|N(U(s'))\|_{\mathcal{H}}
\,ds'
\\
&\leq M'|t_1-t_2|.
\end{align*}

\noindent
{\bf (contractivity)} If $U_1=(u_1,v_1), U_2=(u_2,v_2)\in X_T$, then
\begin{align*}
\|\mathcal{N}(U_1(t))-\mathcal{N}(U_2(t))\|_{\mathcal{H}}^2
&=
\left\|
\begin{pmatrix}0\\|u_1(t)|^p-|u_2(t)|^p\end{pmatrix}
\right\|_{\mathcal{H}}^2
\\
&\leq 
p^2\int_{\R^N}(|u_1(t)|+|u_2(t)|)^{2(p-1)}|u_1(t)-u_2(t)|^2\,dx
\\
&\leq 
p^2\big\||u_1(t)|+|u_2(t)|\big\|_{L^{N(p-1)}}^{2(p-1)}
\|u_1(t)-u_2(t)\|_{L^{\frac{2N}{N-2}}}^{2}
\\
&\leq 
C(\|u_1(t)\|_{H^2}+\|u_2(t)\|_{H^2})^{2(p-1)}
\|u_1(t)-u_2(t)\|_{H^1}^{2}
\\
&\leq 
(C')^2M^{2(p-1)}
\|U_1(t)-U_2(t)\|_{\mathcal{H}}^{2}.
\end{align*}
This implies that 
\begin{align*}
\|\Psi U_1(t)-\Psi U_2(t)\|_{\mathcal{H}}
&\leq
\int_{0}^t\|T_\kappa(t-s)[\mathcal{N}(U_1(s))-\mathcal{N}(U_2(s))]\|_{\mathcal{H}}\,ds
\\
&\leq
C'M^{p-1}e^{T/2}\int_{0}^t\|U_1(s)-U_2(s)\|_{\mathcal{H}}\,ds.
\end{align*}
Consequently, taking $T_0\in (0,T_1]$ satisfying 
\[
2C'M^{p-1}Te^{T/2}\leq 1, 
\]
we obtain $d(\Psi U_1,\Psi U_2)\leq \frac{1}{2}d(U_1,U_2)$, that is, $\Psi$ is  contractive in $X_{T_0}$ and also in $Y_{T_0}$.
\end{proof}

\begin{proof}[Proof of Proposition \ref{prop:solve}]
By Lemma \ref{lem:mild}, we can find a unique fixed point $U_\infty$ of $\Psi$ in $Y_{T_0}$.
Moreover, combining the previous arguments implies 
\begin{align*}
\|\mathcal{N}(U_\infty(t_1))-\mathcal{N}(U_\infty(t_2))\|_{\mathcal{H}}
&\leq 
C'M^{p-1}
\|U_\infty(t_1)-U_\infty(t_2)\|_{\mathcal{H}}
\\
&\leq 
C''M^{p}|t_1-t_2|.
\end{align*}
Thus $\mathcal{N}(U_\infty(\cdot))$ is Lipschitz continuous on $[0,T]$. 
By \cite[Corollary 4.2.11(p.109)]{Pazybook}, we verify that 
$U_t+\mathcal{A}_\kappa U=F(U_\infty)$ has a unique strong solution $U_\infty^*$ given by 
\[
U_{\infty}^*(t)=T_\kappa(t)U_0
+
\int_{0}^tT_\kappa(t-s)\mathcal{N}(U_\infty(s))\,ds, \quad t\in [0,T_0].
\]
Since $U_\infty$ is a fixed point of $\Psi$, we obtain $U_\infty^*(t)=U_\infty(t)$ 
for $t\in [0,T_0]$. Since $U_\infty$ is a strong solution, 
we have $\pa_t U_\infty=-\mathcal{A}_\kappa U_\infty+\mathcal{N}(U_\infty)\in L^\infty(0,T;\mathcal{H})$ a.e.\ on $[0,T]$. This gives us that 
\[
\pa_tu=v\in L^{\infty}(0,T;H^2(\R^N))\cap  W^{1,\infty}(0,T;H^1(\R^N))\cap C([0,T];H^1(\R^N)), 
\]
and 
\[
\pa_tv=\Delta u-\frac{\kappa}{|x|}v+|u|^p\in L^{\infty}(0,T;H^1(\R^N))
\cap C([0,T];L^2).
\]
Hence $u\in C^2([0,T];L^2(\R^N))$ is nothing but a strong solution of 
\[
\pa_t^2u-\Delta u +\frac{\kappa}{|x|}\pa_tu=|u|^p
\]
on $[0,T]$. Uniqueness of local solutions is due to 
a proof similar  to the contractivity of $\Psi$ and 
the finite propagation property follows from the use of $Y_{T_0}$.  
\end{proof}

\section{Special solutions of linear damped wave equation}

In this section we construct special solutions of linear damped wave equation
which will be test functions for proving blowup properties. 

The following function plays an essential role in the proof of upper bound
of lifespan of solutions to \eqref{ndw}.
Similar test functions appear in Zhou--Han \cite{ZH14}. 

\begin{definition}\label{def:phi}
For $\beta>0$, set 
\[
\Phi_{\beta}(x,t)
=
(|x|+t)^{-\beta}F\left(\beta,\frac{N-1+V_0}{2},N-1;\frac{2|x|}{2+t+|x|}\right),
\]
where $F(a,b,c;z)$ is the Gauss hypergeometric function given by 
\begin{align*}
F(a,b,c;z)=\sum_{n=0}^\infty\frac{(a)_n(b)_n}{(c)_n}\,\frac{z^n}{n!}
\end{align*}
with $(d)_0=1$ and $(d)_n=\prod_{k=1}^n (d+k-1)$ for $n\in\N$.
(For further properties of $F(\cdot,\cdot,\cdot;z)$, see 
e.g., Chaper 8 in Beals--Wong \cite{BW}).
\end{definition}

For the reader's convenience we would give a derivation 
the Gauss hypergeometric function from the wave equation. 

\begin{lemma}\label{lem:textfunctions}
For $\beta>0$, $\Phi_\beta$ satisfies the wave equation with the anti-damping term
\[
\pa_t^2\Phi_{\beta}-\Delta \Phi_{\beta}-\frac{V_0}{|x|}\pa_t \Phi_{\beta}=0, 
\quad \text{in}\ 
\mathcal{Q}=\{(x,t)\in\R^N\times (0,\infty)\;;\;|x|<2+t\}. 
\]
\end{lemma}

\begin{proof}
We can put $\Phi(x,t)=\Phi_\beta(x,t-2)$ for $t>0$. 
We start with the desired equation 
\begin{equation}\label{dw}
\pa_t^2 \Phi(x,t) -\Delta \Phi(x,t) - \frac{V_0}{|x|}\pa_t \Phi(x,t)=0, \quad\text{in}\ 
\{(x,t)\in\R^N\times (0,\infty)\;;\;|x|<t\}. 
\end{equation}
Put 
\[
u(x,t)=(|x|+t)^{-\beta}\varphi\left(\frac{2|x|}{|x|+t}\right), 
\]
Then setting $z=\frac{2|x|}{|x|+t}=2-\frac{2t}{|x|+t}$, we have $|x|+t=\frac{2t}{2-z}$ and therefore
\[
u(x,t)=
(2t)^{-\beta}(2-z)^{\beta}\varphi(z).
\]
Observing that 
\[
\frac{\pa z}{\pa t} = -\frac{2|x|}{(|x|+t)^2}=-\frac{z(2-z)}{2t}, 
\]
we have
\begin{align*}
\pa_t u 
&=-2\beta(2t)^{-\beta-1}(2-z)^{\beta}\varphi(z)
+
(2t)^{-\beta}[-\beta(2-z)^{\beta-1}\varphi(z)+(2-z)^{\beta}\varphi'(z)]\frac{\pa z}{\pa t}
\\
&=-2\beta(2t)^{-\beta-1}(2-z)^{\beta}\varphi(z)
-
(2t)^{-\beta-1}[-\beta(2-z)^{\beta-1}\varphi(z)+(2-z)^{\beta}\varphi'(z)]z(2-z).
\\
&=
-
(2t)^{-\beta-1}(2-z)^{\beta+1}
\Big[
\beta\varphi(z)
+z\varphi'(z)
\Big]
\end{align*}
and also
\begin{align*}
\pa_t^2 u
&=
(2t)^{-\beta-2}(2-z)^{\beta+2}
\Big[
(\beta+1)(\beta\varphi(z)+z\varphi'(z))
+z(\beta\varphi(z)+z\varphi'(z))'
\Big]
\\
&=
(2t)^{-\beta-2}(2-z)^{\beta+2}
\Big[
\beta(\beta+1)\varphi(z)+2(\beta+1)z\varphi'(z)+z^2\varphi''(z)
\Big].
\end{align*}
On the other hand, for radial derivative, we see from $\frac{\pa z}{\pa r} = \frac{2t}{(|x|+t)^2}=\frac{(2-z)^2}{2t}$ that 
\begin{align*}
\pa_ru
&=(2t)^{-\beta}\Big[-\beta(2-z)^{\beta-1}\varphi(z)+(2-z)^{\beta}\varphi'(z)\Big]\frac{\pa z}{\pa r}
\\
&=(2t)^{-\beta-1}(2-z)^{\beta+1}\Big[-\beta\varphi(z)+(2-z)\varphi'(z)\Big]
\end{align*}
and 
\begin{align*}
\pa_r^2u
&=
(2t)^{-\beta-2}(2-z)^{\beta+2}
\Big[
(\beta+1)\beta\varphi(z)-2(\beta+1)(2-z)\varphi'(z)
+(2-z)^2\varphi''(z)
\Big].
\end{align*}
Combining these equalities and $\frac{t}{r}(2-z)^{-1}=\frac{1}{z}$, we obtain
\begin{align*}
0&=
\left(\pa_t^2u-\pa_r^2u-\frac{N-1}{r}\pa_ru+\frac{V_0}{r}\pa_tu\right)
(2t)^{\beta+2}(2-z)^{-\beta-2}
\\
&=
\beta(\beta+1)\varphi(z)+2(\beta+1)z\varphi'(z)+z^2\varphi''(z)
-(\beta+1)\beta\varphi(z)-2(\beta+1)(2-z)\varphi'(z)+(2-z)^2\varphi''(z)
\\
&\quad\quad-\frac{N-1}{r}(2t)(2-z)^{-1}
\Big[-\beta\varphi(z)+(2-z)\varphi'(z)\Big]
-
\frac{V_0}{r}(2t)(2-z)^{-1}
\Big[
-\beta\varphi(z)
-z\varphi'(z)
\Big]
\\
&=
-4(1-z)\varphi''(z)
+4(\beta+1)\varphi'(z)
+\frac{2\beta(N-1+V_0)}{z}\varphi(z)
+\frac{2(N-1)}{z}
\Big[-(2-z)\varphi'(z)\Big]
-
2V_0\varphi'(z)
\\
&=-\frac{4}{z}
\left[
(1-z)z\varphi''(z)
+\left[N-1
-\left(1+\beta+\frac{N-1+V_0}{2}\right)z\right]\varphi'(z)
-\frac{\beta(N-1+V_0)}{2}\varphi(z)
\right].
\end{align*}
This is nothing but the Gauss hypergeometric differential equation 
\[
z(1-z)\varphi''(z)+(c-(1+a+b)z)\varphi'(z)-ab\varphi(z)=0
\]
with 
\[
(a,b,c)=
\left(\beta, \frac{N-1+V_0}{2}, N-1\right).
\]
This implies that $\varphi(z)=F(\beta,\frac{N-1+V_0}{2},N-1;z)$. 
\end{proof}

\begin{lemma}\label{lem:phi.prop}
\noindent
{\bf (i)} For every $\beta>0$ and $(x,t)\in \mathcal{Q}$, 
\[
\pa_t\Phi_\beta(x,t)=-\beta\Phi_{\beta+1}(x,t).
\]
\noindent
{\bf (ii)} If $0<\beta<\frac{N-1-V_0}{2}$, 
then there exists a constant 
$c_\beta>0$ such that for every $(x,t)\in\mathcal{Q}$, 
\[
c_\beta(2+t)^{-\beta}
\leq 
\Phi_\beta(x,t)
\leq 
c_\beta^{-1}(2+t)^{-\beta}. 
\]
\noindent
{\bf (iii)} If $\beta>\frac{N-1-V_0}{2}$, then 
there exists a constant $c_\beta'>0$ such that for every $(x,t)\in\mathcal{Q}$, 
\[
c_\beta(2+t)^{-\beta}
\left(1-\frac{|x|}{t+2}\right)^{\frac{N-1-V_0}{2}-\beta}
\leq 
\Phi_\beta(x,t)
\leq 
c_\beta^{-1}(2+t)^{-\beta}
\left(1-\frac{|x|}{t+2}\right)^{\frac{N-1-V_0}{2}-\beta}. 
\]
\end{lemma}

\begin{proof}
{\bf (i)} In view of the proof of Lemma \ref{lem:textfunctions}, 
we have 
\[
\pa_t \Phi_{\beta}(x,t)=-(2t)^{-\beta-1}(2-z)^{\beta+1}[\beta \varphi(z) + z\varphi'(z)]
\]
with $s=\frac{2|x|}{2+t+|x|}$. It suffices to show that 
\begin{equation}\label{gauss.aux}
\beta \varphi(z) + z\varphi'(z)=
\beta 
F\left(\beta+1,\frac{N-1+V_0}{2},N-1;z\right),
\quad 
 z\in (0,1).
\end{equation}
Put $\psi(z)=\beta \varphi(z) + z\varphi'(z)$ for $z\in [0,1)$. 
Then by the definition of $F(\cdot,\cdot,\cdot;z)$, 
we have $\psi(0)=\beta$. On the other hand, we see from 
the gauss hypergeometric equation with 
$a=\beta$, $b=\frac{N-1+V_0}{2}$ and $c=N-1$
that 
\begin{align*}
(1-z)\psi'(z)
&=(1-z)\Big((\beta+1) \varphi'(z) + z\varphi''(z)\Big)
\\
&=(\beta+1)(1-z)\varphi'(z) + z(1-z)\varphi''(z)
\\
&=(a+1)(1-z)\varphi'(z) 
- (c-(1+a+b)z)\varphi'(z)+ab\varphi(z)
\\
&=(a+1-c)\varphi'(z) 
+bz\varphi'(z)+ab\varphi(z)
\\
&=(a+1-c)\varphi'(z) 
+b\psi(z)
\end{align*}
and therefore $(1-z)\psi'(z)-b\psi(z)=(a+1-c)\varphi'(z)$. 
The definition of $\psi$ yields 
\begin{align*}
z(1-z)\psi'(z)-bz\psi(z)
&=(a+1-c)z\varphi'(z)
\\
&=(a+1-c)\psi(z)-(a+1-c)a\varphi(z)
\end{align*} 
Differentiating the above equality, we have
\begin{align*}
z(1-z)\psi''(z)+(1-(2+b)z)\psi'(z)-b\psi(z)
&=(a+1-c)\psi'(z)-(a+1-c)a\varphi'(z)
\\
&=(a+1-c)\psi'(z)-a\Big((1-z)\psi'(z)-b\psi(z)\Big).
\end{align*}
Hence we have $z(1-z)\psi''(z)+(c-(2+a+b)z)\psi'(z)+(a+1)b\psi(z)=0$. 
Since $N\geq 2$, all bound solutions of this equation near $0$ 
can be written by $\psi(z)=h F(a+1,b,c;z)$ with $h\in \R$. 
Combining the initial value $\psi(0)=\beta$, we obtain \eqref{gauss.aux}. 

The remaining assertions {\bf (ii)} and {\bf (iii)} are a direct consequence of the integral representation formula
\begin{align*}
F(a,b,c,z)
&=
\frac{1}{B(c,c-a)}
\int_{0}^{1}
  s^{a-1}(1-s)^{c-a-1}(1-zs)^{-b}\,ds, 
  \quad 
  0\leq z<1
\end{align*}
when $c>0$ and $c-a>0$. The proof is complete. 
\end{proof}

\section{Proof of blowup phenomena}

In this section we prove upper bound of the lifespan of solutions 
to \eqref{ndw} and its dependence of $\ep$
under the condition $0\leq V_0<V_*=\frac{(N-1)^2}{N+1}$. 

\subsection{Preliminaries for showing blowup phenomena}
We first state a criterion for derivation of 
upper bound for lifespan. 
\begin{lemma}\label{lem:blowup}
Let $H\in C^2([\sigma_0,\infty)$ be nonnegative function. 

\noindent{\bf (i)}
Assume that there exists positive constants $c,C,C'$ such that
\begin{align*}
C[H(\sigma)]^{p}\leq H''(\sigma)+c\,\frac{H'(\sigma)}{\sigma}
\end{align*}
with $H(\sigma)\geq \ep^pC\sigma^2$ and $H'(\sigma)\geq \ep^p C\sigma$. Then 
$H$ blows up before $\sigma=C''\ep^{-\frac{p-1}{2}}$ for some $C''>0$. 

\noindent{\bf (ii)}
Assume that there exists positive constants $c,C,C'$ such that
\begin{align*}
C\sigma^{1-p}[H(\sigma)]^{p}\leq H''(\sigma)+2 H'(\sigma)
\end{align*}
with $H(\sigma)\geq \ep^pC\sigma$ and $H'(\sigma)\geq \ep^p C$. Then 
$H$ blows up before $\sigma=C''\ep^{-p(p-1)}$ for some $C''>0$. 
\end{lemma}

\begin{proof}
The assertion follows from \cite[Lemma 2.1]{ZH14} with 
the argument in \cite[Section 3]{ZH14}.
\end{proof}

We focus our eyes to the following functionals. 
\begin{definition}\label{def:G_beta}
For $\beta\in (0,\frac{N-1-V_0}{2})$, define the following three functions
\begin{align*}
G_\beta(t)
&:=
\int_{\R^N}
  |u(x,t)|^{p}
  \Phi_{\beta}(x,t)
\,dx,\quad t\geq 0,
\\
H_\beta(t)
&:=
\int_0^t
  (t-s)(2+s)G_\beta(t)
\,ds, 
\quad t\geq 0,
\\
J_\beta(t)
&:=
\int_0^t
  (2+s)^{-3}H_\beta(t)
\,ds, 
\quad t\geq 0. 
\end{align*}
Note that we can see from Lemma \ref{lem:phi.prop} {\bf (ii)} that 
$G_\beta(t)\approx (2+t)^{-\beta}\|u(t)\|_{L^p(\R^N)}^p$. 
\end{definition}

\begin{lemma}\label{lem:sufficient}
If $u_\ep$ is a solution of \eqref{ndw} in Proposition \ref{prop:solve}
with parameter $\ep>0$, 
then $J_\beta$ does not blow up until $\lifespan (u_\ep)$. 
\end{lemma}
\begin{proof}
It follows from the embedding $H^2(\R^N)\to L^p(\R^N)$ 
(given by Gagliardo-Nirenberg-Sobolev inequalities) that 
$\|u_\ep(t)\|_{L^p}$ is continuous on $[0,\lifespan (u_\ep))$ 
and also $G_\beta(t)$. 
This means that $J_\beta(t)$ is finite for all 
$t\in [0,\lifespan (u_\ep))$. 
\end{proof}

\begin{lemma}\label{lem:trick}
For every $\beta>0$ and $t\geq 0$, 
\[
(2+t)^2J_\beta(t)=\frac{1}{2}\int_0^t(t-s)^2G_\beta(s)\,ds.
\]
\end{lemma}
\begin{proof}
This can be verified by integration by parts twice, by noting that 
\[
\frac{d}{ds}\Big((t-s)^2(1+s)^{-1}\Big)=\frac{2(1+t)^2}{(1+s)^3}. 
\]
\end{proof}

\begin{lemma}\label{base}
Let $u$ be a solution of \eqref{ndw}. Then for every $\beta>0$ and $t\geq0$, 
\begin{align}
\nonumber
\ep 
E_{\beta,0}
+
\ep 
E_{\beta,1}t
+
\int_0^t(t-s)G_\beta(s)\,ds
&=
\int_{\R^N}
   u(x,t)\Phi_\beta(x,t)
\,dx
+
2\beta
\int_0^t
\int_{\R^N}
   u(x,s)\Phi_{\beta+1}(x,s)
\,dx
\,ds
\\
\label{G-equality}
&\quad +
V_0
\int_0^t
\int_{\R^N}
   \frac{1}{|x|}u(x,t)\Phi_\beta(x,t)
\,dx
\,ds, 
\end{align}
where 
\begin{align*}
E_{\beta,0}
&=
\int_{\R^N}
   f(x)\Phi_\beta(x,0)
\,dx>0.
\\
E_{\beta,1}&=\int_{\R^N}
   g(x)\Phi_\beta(x,0)
\,dx
+
\beta\int_{\R^N}
   f(x)\Phi_{\beta+1}(x,0)
\,dx
+
V_0\int_{\R^N}
   \frac{1}{|x|}f(x)\Phi_\beta(x)
\,dx>0.
\end{align*}
\end{lemma}

\begin{proof}
By the equation in \eqref{ndw} 
we see from integration by parts that
\begin{align*}
G_\beta(t)
&=
\int_{\R^N}
\left(
  \pa_t^2u(t)-\Delta u(t)+\frac{V_0}{|x|}\pa_tu(t)
\right)\Phi_\beta(t)
\,dx
\\
&=
\int_{\R^N}
\left(
  \pa_t^2u(t)+\frac{V_0}{|x|}\pa_tu(t)
\right)\Phi_\beta(t)
\,dx
-
\int_{\R^N}u(t)(\Delta \Phi_\beta(t))\,dx.
\end{align*}
Using Lemma \ref{lem:textfunctions}, we have
\begin{align*}
G_\beta(t)&=
\int_{\R^N}
\left(
  \pa_t^2u(t)+\frac{V_0}{|x|}\pa_tu(t)
\right)\Phi_\beta(t)
\,dx
-
\int_{\R^N}u(t)
\left(\pa_t^2\Phi_\beta(t)-\frac{V_0}{|x|}\pa_t\Phi_\beta(t)\right)\,dx
\\
&=
\frac{d}{dt}
\left[
\int_{\R^N}\left(\pa_t u(t)\Phi_\beta(t)-u(t)\pa_t\Phi_\beta(t)\right)\,dx
+
V_0\int_{\R^N}\frac{1}{|x|}u(t)\Phi_\beta(t)\,dx\right].
\end{align*}
Noting that Lemma \ref{lem:phi.prop} {\bf (i)} 
(the formula $\pa_t\Phi_\beta=-\beta\Phi_{\beta+1}$), we have
\begin{align*}
\ep E_{\beta,1}+
\int_0^t
G_\beta(s)\,ds
&= 
\int_{\R^N}\left(\pa_t u(t)\Phi_\beta(t)-u(t)\pa_t\Phi_\beta(t)\right)\,dx
+
V_0\int_{\R^N}\frac{1}{|x|}u(t)\Phi_\beta(t)\,dx
\\
&= 
\frac{d}{dt}\left[\int_{\R^N}u(t)\Phi_\beta(t)\,dx\right]
+
2\beta\int_{\R^N}u(t)\Phi_{\beta+1}(t)\,dx
+
V_0\int_{\R^N}\frac{1}{|x|}u(t)\Phi_\beta(t)\,dx.
\end{align*}
Integrating it again, we obtain \eqref{G-equality}.
\end{proof}

The following lemma makes sense when 
$\frac{2}{N-1-V_0}<p_0(N+V_0)$ which is equivalent
to $0\leq V_0<\frac{(N-1)^2}{N+1}$.
\begin{lemma}\label{base2}
Assume $\frac{N}{N-1}<p<\infty$ and $0\leq V_0<\frac{(N-1)^2}{N+1}$.
{\rm (i)}\ Let $q>1$ satisfy $\max\{p,\frac{2}{N-1-V_0}\}<q< \infty$
and put 
\[
\beta=\frac{N-1-V_0}{2}-\frac{1}{q}\in \left(0,\frac{N-1-V_0}{2}\right).
\]
Then there 
exists a positive constant $C_1>0$ such that
\begin{align*}
&\ep E_{\beta,0}+\ep E_{\beta,1} t+
\int_0^t
(t-s)G_\beta(s)\,ds
\leq
C_1
\left[
\|u(t)\|_{L^p}
(2+t)^{\frac{N}{p'}-\beta}
+
\int_0^t
\|u(s)\|_{L^p}
(2+s)^{\frac{N}{p'}-\frac{N-1-V_0}{2}-\frac{1}{p'}}
\,ds
\right].
\end{align*}
{\rm (ii)}\ If $p>\frac{2}{N-1-V_0}$, then setting 
$
\beta_0=\frac{N-1-V_0}{2}-\frac{1}{p}\in \left(0,\frac{N-1-V_0}{2}\right),
$
one has
\begin{align*}
&
\int_0^t
(t-s)G_\beta(s)\,ds
\leq
C_1
\left[
\|u(t)\|_{L^p}
(2+t)^{\frac{N}{p'}-\beta}
+
\int_0^t
\|u(s)\|_{L^p}
(2+s)^{\frac{N}{p'}-\beta_0-1}(\log (2+s))^{\frac{1}{p'}}
\,ds
\right].
\end{align*}

\end{lemma}
\begin{proof}
By Lemma \ref{base} with finite propagation property, we have 
\begin{align*}
&
\ep E_{\beta,0}+\ep E_{\beta,1} t+
\int_0^t
(t-s)G_\beta(s)\,ds
=I_{\beta,1}(t)+2\beta I_{\beta,2}(t)+V_0 I_{\beta,3}(t).
\end{align*}
where 
\begin{align*}
I_{\beta,1}(t)
&=
\int_{B(0,1+t)}u(x,t)\Phi_\beta(x,t)\,dx
\\
I_{\beta,2}(t)
&=
\int_{0}^t\left(\int_{B(0,1+t)}u(x,s)\Phi_{\beta+1}(x,s)\,dx\right)\,ds
\\
I_{\beta,3}(t)
&=
\int_0^t\left(\int_{B(0,1+t)}\frac{1}{|x|}u(x,s)\Phi_\beta(x,s)\,dx\right)\,ds. 
\end{align*}
Using Lemma \ref{lem:phi.prop} {\bf (ii)}, we have 
\begin{align*}
I_{\beta,1}(t)
&\leq 
\left(\int_{B(0,1+t)}|u(x,t)|^p\,dx\right)^\frac{1}{p}
\left(\int_{B(0,1+t)}\Phi_\beta(x,t)^{p'}\,dx\right)^\frac{1}{p'}
\\
&\leq 
c_\beta^{-1}N^{-\frac{1}{p'}}|S^{N-1}|^{\frac{1}{p'}}\|u(t)\|_{L^p}(2+t)^{\frac{N}{p'}-\beta}.
\end{align*}
and 
\begin{align*}
I_{\beta,3}'(t)
&\leq 
\left(\int_{B(0,1+t)}|u(x,t)|^p\,dx\right)^\frac{1}{p}
\left(\int_{B(0,1+t)}\frac{1}{|x|^{p'}}\Phi_\beta(x,t)^{p'}\,dx\right)^\frac{1}{p'}
\\
&\leq 
c_\beta^{-1}(N-p')^{-\frac{1}{p'}}|S^{N-1}|^{\frac{1}{p'}}
\int_0^t \|u(t)\|_{L^p}(2+t)^{\frac{N}{p'}-\beta-1}\,ds.
\end{align*}
Noting that $\beta+1=\frac{N-1-V_0}{2}-\frac{1}{q'}$, 
we see from Lemma \ref{lem:phi.prop} {\bf (iii)} that 
\begin{align*}
I_{\beta,2}'(t)
&\leq 
\left(\int_{B(0,1+t)}|u(x,t)|^p\,dx\right)^\frac{1}{p}
\left(\int_{B(0,1+t)}\Phi_{\beta+1}(x,t)^{p'}\,dx\right)^\frac{1}{p'}
\\
&\leq 
(c'_\beta)^{-1}\|u(t)\|_{L^p}
(2+t)^{-\beta-1}
\left(
   \int_{B(0,1+t)}
   \left(1-\frac{|x|}{t+2}\right)^{(\frac{N-1-V_0}{2}-\beta-1)p'}\,dx\right)^\frac{1}{p'}
\\
&= 
(c'_\beta)^{-1}|S^{N-1}|^{\frac{1}{p'}}\|u(t)\|_{L^p}
(2+t)^{-\beta-1}
\left(
   \int_0^{1+t}\left(1-\frac{r}{t+2}\right)^{-\frac{p'}{q'}}r^{N-1}\,dr\right)^\frac{1}{p'}
\\
&= 
(c'_\beta)^{-1}|S^{N-1}|^{\frac{1}{p'}}\|u(t)\|_{L^p}
(2+t)^{\frac{N}{p'}-\beta-1}
\left(
   \int_{\frac{1}{2+t}}^1\rho^{-\frac{p'}{q'}}(1-\rho)^{N-1}\,d\rho\right)^\frac{1}{p'}
\\
&\leq  
(c'_\beta)^{-1}|S^{N-1}|^{\frac{1}{p'}}\left(
   \frac{p'}{q'}-1\right)^\frac{1}{p'}
\|u(t)\|_{L^p}
(2+t)^{\frac{N}{p'}-\beta-1+\frac{1}{q'}-\frac{1}{p'}}. 
\end{align*}
Thus we have 
\begin{align*}
&\ep E_{\beta,0}+\ep E_{\beta,1} t+
\int_0^t
(t-s)G_\beta(s)\,ds
\leq
C_1
\left[
\|u(t)\|_{L^p}
(2+t)^{\frac{N}{p'}-\beta}
+
\int_0^t
\|u(s)\|_{L^p}
(2+s)^{\frac{N}{p'}-\beta-1+\frac{1}{q'}-\frac{1}{p'}}
\,ds
\right].
\end{align*}
By the definition of $\beta$ we have the first desired inequality. 
The second is verified by noticing $q'/p'=1$ in the previous proof. 
\end{proof}

\subsection{Proof of Theorem \ref{main} for subcritical case 
$\max\{\frac{N}{N-1},\frac{N+3+V_0}{N+1+V_0}\}<p<p_0(N+V_0)$}

\begin{proof}
Fix $q>p$ as the following way:
\[
\frac{1}{q}\in 
\left(0,\frac{N-1-V_0}{2}\right)
\cap 
\left(\frac{(N-1+V)p-(N+1+V)}{2}, \frac{(N+1+V_0)p-(N+3+V_0)}{2}\right).
\]
The above set is not empty when 
$(p,V_0)\in \Omega_1\cup\Omega_2\cup\Omega_3$; 
note that for respective cases we can take 
\[
\frac{1}{q}=\begin{cases}
\frac{1}{p}-\delta
&
\text{if }(p,V_0)\in\Omega_1,
\\
\frac{N-1-V_0}{2}-\delta
&
\text{if }(p,V_0)\in\Omega_2,
\\
\frac{(N+1+V_0)p-(N+3+V_0)}{2}-\delta
&
\text{if }(p,V_0)\in\Omega_3
\end{cases}
\]
with arbitrary small $\delta>0$. 
Moreover, this condition is equivalent to 
\[
q>p,\quad
\beta=\frac{N-1-V_0}{2}-\frac{1}{q}>0,
\quad
\&\quad
\lambda=\frac{\gamma(N+V_0;p)}{2p}-\frac{1}{p}+\frac{1}{q}\in (0,p-1).
\]
Then we see by Lemma \ref{base2} {\bf (i)} that 
\begin{align}\label{eq:4.2.1}
&
\ep E_{\beta,0}+\ep E_{\beta,1} t+
\int_0^t
(t-s)G_\beta(s)\,ds
\leq
C_1'
\left[
G_\beta(t)^\frac{1}{p}
(2+t)^{\frac{N-\beta}{p'}}
+
\int_0^t
G_\beta(s)^\frac{1}{p}
(2+s)^{\frac{N-\beta}{p'}-\frac{1}{q}-\frac{1}{p'}}
\,ds
\right].
\end{align}
Observe that 
\begin{align*}
\frac{N-\beta}{p'}-\frac{1}{q}-\frac{1}{p}
&=
\frac{1}{p}\left(p-1
-\lambda\right)
>0, 
\end{align*}
Integrating \eqref{eq:4.2.1} over $[0,t]$, we deduce
\begin{align*}
&
\ep E_{\beta,0}t+\ep \frac{E_{\beta,1}}{2} t^2+
\frac{1}{2}\int_0^t
(t-s)^2G_\beta(s)\,ds
\\
&\leq
C_1'
\left[
\int_0^t G_\beta(s)^\frac{1}{p}
  (2+s)^{\frac{N-\beta}{p'}}
\,ds
+
\int_0^t
(t-s)G_\beta(s)^\frac{1}{p}
(2+s)^{\frac{N-\beta}{p'}-\frac{1}{q}-\frac{1}{p'}}
\,ds
\right]
\\
&\leq
C_1'
\left(
\int_0^t(2+s)G_\beta(s)\,ds
\right)^{\frac{1}{p}}
\left[
\left(
\int_0^t(2+s)^{(\frac{N-\beta}{p'}-\frac{1}{p})p'}\,ds
\right)^{\frac{1}{p'}}
+
\left(
\int_0^t(t-s)^{p'}(2+s)^{(\frac{N-\beta}{p'}-\frac{1}{q}-\frac{1}{p})p'-1}\,ds
\right)^{\frac{1}{p'}}
\right]
\\
&\leq
C_2
\left(
\int_0^t(2+s)G_\beta(s)\,ds
\right)^{\frac{1}{p}}
\left[
(2+t)^{\frac{N-\beta}{p'}-\frac{1}{p}+\frac{1}{p'}}
+
(2+t)^{1+\frac{N-\beta}{p'}-\frac{1}{q}-\frac{1}{p}}
\right]
\\
&\leq
2C_2
\left(
\int_0^t(2+s)G_\beta(s)\,ds
\right)^{\frac{1}{p}}
(2+t)^{1+\frac{p-1-\lambda}{p}}.
\end{align*}
We see from the definition of $H_\beta$ that
\begin{align*}
(2C_2)^{-p}(2+t)^{1+\lambda-2p}\left(
\ep E_{\beta,0}t+\ep \frac{E_{\beta,1}}{2} t^2
\right)^p
&\leq H_\beta'(t).
\end{align*}
Hence
\[
H_\beta'(t)\geq 
C_4\ep^{p}(2+t)^{1+\lambda}, \quad t\geq 1. 
\] 
Integrating it over $[0,t]$, we have for $t\geq 2$, 
\[
H_\beta(t)
\geq 
\int_0^t H_\beta'(s)\,ds
\geq 
\int_1^t H_\beta'(s)\,ds
\geq 
C_4\ep^{p}\int_1^t (2+s)^{\lambda}\,ds
\geq 
\frac{C_4\ep^p}{4(2+\lambda)}(2+t)^{2+\lambda}. 
\]
We see from the definition of $I_\beta$ that for $t\geq 2$,
\[
J_\beta'(t)=(2+s)^{-3}H_\beta(t)\geq\frac{C_4\ep^p}{4(2+\lambda)}(2+t)^{-1+\lambda}.
\] 
and for $t\geq 4$,
\[
J_\beta(t)=
\int_2^t J_\beta'(s)\,ds
\geq 
\frac{C_4\ep^p}{8\lambda(2+\lambda)}
(2+t)^{\lambda}.
\]
On the other hand, 
we see from Lemma \ref{lem:trick} that
\[
(2C_2)^{-p}[J_\beta(t)]^p\leq 
(2+t)^{2-\lambda}J_\beta''(t)+3(2+t)^{1-\lambda} J_\beta'(t)]. 
\]
Moreover, setting 
$J_\beta(t)=\widetilde{J}_\beta(\sigma)$, $\sigma=\frac{2}{\lambda}(2+t)^{\frac{\lambda}{2}}$,  we see 
\[
(2+t)^{1-\frac{\lambda}{2}}J_\beta'(t)=\widetilde{J}_\beta'(\sigma), 
\quad 
(2+t)^{2-\lambda}J_\beta''(t)
+\frac{2-\lambda}{2}(2+t)^{1+\lambda}J_\beta'(t)
=
\widetilde{J}_\beta''(\sigma).
\]
Then 
\begin{gather*}
C_5^{-p}[\widetilde{J}_\beta(\sigma)]^p
\leq 
\widetilde{J}_\beta''(\sigma)
+
\frac{4+\lambda}{\lambda}\sigma^{-1} 
\widetilde{J}_\beta'(\sigma), 
\quad 
\sigma\geq \sigma_0
=
\frac{2}{\lambda}2^{\frac{\lambda}{2}}, 
\\
\widetilde{J}_\beta'(\sigma)
\geq 
C_6\ep^p\sigma, 
\quad
\sigma\geq \sigma_1=\frac{2}{\lambda}4^{\frac{\lambda}{2}}, 
\\
\widetilde{J}_\beta(\sigma)
\geq 
C_6\ep^p\sigma^2, 
\quad
\sigma\geq \sigma_2=\frac{2}{\lambda}6^{\frac{\lambda}{2}}.
\end{gather*}
Consequently, by Lemma \ref{lem:blowup} {\bf (i)} we see that 
$\widetilde{J}_\beta$ blows up before $C_7\ep^{-\frac{p-1}{2}}$ and then, 
$J_\beta$ blows up before $C_7\ep^{-\frac{p-1}{\lambda}}$. 
By virtue of Lemma \ref{lem:sufficient}, we have 
$\lifespan(u_\ep)\leq C_7\ep^{-\frac{p-1}{\lambda}}$.

Finally, we remark that 
if $(p,V_0)\in\Omega_1$, then we can take $1/q=1/p-\delta$ 
for arbitrary small $\delta>0$ and then
$\lambda=\gamma(N+V_0;p)/(2p)-\frac{1}{p}+\frac{1}{q}
=\gamma(N+V_0;p)/(2p)-\delta$. 
This implies that 
\[
\lifespan u
\leq 
C_7\ep^{-\frac{2p(p-1)}{\gamma(N+V_0;p)}-\delta'}.
\]
for arbitrary small $\delta'>0$. The proof is complete.
\end{proof}

\subsection{Proof of Theorem \ref{main} for critical case 
$p=p_0(N+V_0)$}
\begin{proof}
In this case we set
\[
\beta_\delta=\frac{N-1-V_0}{2}-\frac{1}{p+\delta}\in \left(0,\frac{N-1-V_0}{2}\right).
\]
Then by Lemma \ref{base2} {\bf (ii)} with $\beta=\beta_\delta$, 
\begin{align*}
\ep E_{\beta,0}+\ep E_{\beta,2} t
&\leq 
C_1
\left[
\|u(t)\|_{L^p}
(2+t)^{\frac{N}{p'}-\beta}
+
\int_0^t
\|u(s)\|_{L^p}
(2+s)^{\frac{N}{p'}-\beta_0-1}
\,ds
\right]
\\
&\leq 
K_1
\left[
\left(G_{\beta_0}(t)\right)^{\frac{1}{p}}
(2+t)^{\frac{N-\beta_0}{p'}+(\beta_0-\beta)}
+
\int_0^t
\left(G_{\beta_0}(t)\right)^{\frac{1}{p}}
(2+s)^{\frac{N-\beta_0}{p'}-1}
\,ds
\right].
\end{align*}
Noting that $\frac{N-\beta_0}{p'}=1+\frac{1}{p}$ and  
integrating it over $[0,t]$, we have 
\begin{align*}
\ep E_{\beta_\delta,0}t+\ep \frac{E_{\beta_\delta,1}}{2} t^2
&\leq 
K_1
\left[
\int_0^t \left(G_{\beta_0}(s)\right)^{\frac{1}{p}}
(2+s)^{1+\frac{1}{p}+(\beta_0-\beta_\delta)}\,ds
+
\int_0^t
(t-s)\left(G_{\beta_0}(t)\right)^{\frac{1}{p}}
(2+s)^{\frac{1}{p}}
\,ds
\right]
\\
&\leq 
K_1
\left(
\int_0^tG_{\beta_0}(s)(1+s)\,ds
\right)^{\frac{1}{p}}
\left[
\left(
\int_0^t 
(2+s)^{p'+(\beta_0-\beta_\delta)p'}\,ds
\right)^{\frac{1}{p'}}
+
\left(
\int_0^t
(t-s)^{p'}
\,ds
\right)^{\frac{1}{p'}}
\right]
\\
&\leq 
K_2
\left(
\int_0^tG_{\beta_0}(s)(1+s)\,ds
\right)^{\frac{1}{p}}
(2+t)^{1+\frac{1}{p'}}.
\end{align*}
By the definition of $H_{\beta_0}$, we have
for $t\geq 1$, 
\[
H'_{\beta_0}(t)
\geq 
K_2^{-p}
\ep^p\left(E_{\beta_\delta,0}t+\frac{E_{\beta_\delta,1}}{2} t^2\right)^p
(2+t)^{1-2p}
\geq K_3\ep^p (2+t)
\]
and then for $t\geq 2$, 
\[
H_{\beta_0}(t)\geq \int_1^t\tilde{G}'_{\beta_0}(s)\,ds
\geq 
K_4\ep^p (2+t)^2.
\]
On the other hand, by Lemma \ref{base2} {\bf (ii)} we have 
\begin{align*}
&
\int_0^t
(t-s)G_{\beta_0}(s)\,ds
\leq
C_1
\left[
\|u(t)\|_{L^p}
(2+t)^{\frac{N}{p'}-\beta_0}
+
\int_0^t
\|u(s)\|_{L^p}
(2+s)^{\frac{N}{p'}-\beta_0-1}(\log (2+s))^{\frac{1}{p'}}
\,ds
\right].
\end{align*}
Noting $\frac{N-\beta_0}{p'}=1+\frac{1}{p}$ again 
and integrating it over $[0,t]$, we have 
\begin{align*}
&
\frac{1}{2}\int_0^t
(t-s)G_{\beta_0}(s)\,ds
\\
&\leq
K_1'
\left[
\int_0^t
G_{\beta_0}(s)^\frac{1}{p}
(2+s)^{\frac{N-\beta_0}{p'}}\,ds
+
\int_0^t
(t-s)G_{\beta_0}(s)^\frac{1}{p}
(2+s)^{\frac{N-\beta_0}{p'}-1}(\log (2+s))^{\frac{1}{p'}}
\,ds
\right]
\\
&\leq 
K_1'H_{\beta_0}'(t)^{\frac{1}{p}}
\left[
\int_0^t
(2+s)^{p'}\,ds
+
\int_0^t
(t-s)^{p'}
\log (2+s)
\,ds
\right]
\\
&\leq 
K_2'H_{\beta_0}'(t)^{\frac{1}{p}}
(2+t)^{1+\frac{1}{p'}}(\log (2+t))^\frac{1}{p'}.
\end{align*}
As in the proof of subcritial case, 
we deduce
\begin{align*}
(K_2')^{-p}(\log(2+t))^{1-p}J_{\beta_0}(t)^{p}
&\leq H_{\beta_0}'(t)(1+t)^{-1}
\\
&\leq (2+t)^2 J_{\beta_0}''(t)+3(2+t) J_{\beta_0}'(t).
\end{align*}
Here we take $J_{\beta_0}(t)=\widetilde{J}_{\beta_0}(\sigma)$ 
with $\sigma =\log (2+t)$. Since 
\[
(2+t)J_{\beta_0}'(t)=\widetilde{J}_{\beta_0}'(\sigma), 
\quad 
(2+t)^2J_{\beta_0}'(t)+(2+t)J_{\beta_0}''(t)=\widetilde{J}_{\beta_0}''(\sigma),
\]
we obtain for $\sigma\geq \sigma_0:=\log 2$, 
\[
(K_2')^{-p}\sigma^{1-p}\widetilde{J}_{\beta_0}(\sigma)^p
\leq \widetilde{J}_{\beta_0}''(\sigma)+2\widetilde{J}_{\beta_0}'(\sigma).
\]
Moreover, we have 
for $\sigma \geq \sigma_1=\log 4$, 
\begin{align*}
\widetilde{J}_{\beta_0}'(\sigma)
&=(2+t)J_{\beta_0}'(t)
\\
&=(2+t)^{-2}H_{\beta_0}(t)
\\
&\geq K_4\ep^p
\end{align*}
and therefore for $\sigma \geq \sigma_2=2\log 4$, 
\[
\widetilde{J}_{\beta_0}(\sigma)\geq \frac{K_4}{2}\ep^p\sigma. 
\]
Applying Lemma \ref{lem:blowup} {\bf (ii)} we deduce that 
$\widetilde{J}_{\beta_0}$ blows up before $\sigma=K_5\ep^{-p(p-1)}$. 
Then by definition $J_{\beta_0}$ blows up before $\exp[K_5\ep^{-p(p-1)}]$. 
Consequently, using Lemma \ref{lem:sufficient}, we obtain 
\[
\lifespan u\leq \exp[K_5\ep^{-p(p-1)}].
\]
The proof is complete. 
\end{proof}

\begin{remark}
In particular. in the proof of Theorem \ref{main} with $p=p_(N+V_0)$, 
we have used two kind of auxiliary parameters 
$1/q=\frac{N-1-V_0}{2}-\frac{1}{p}$
and $1/q=\frac{N-1-V_0}{2}-\frac{1}{p+\delta}$. 
The first choice is for deriving lower bound of the functional 
$J_{\beta_0}$ and the second is for deriving 
differential inequality for $J_{\beta_0}$. 
The first choice is essentially different from 
the idea 
of Yordanov--Zhang \cite{YZ06} to prove the lower bound of a functional.   
\end{remark}

\subsection*{Acknowedgements}
This work is partially supported 
by Grant-in-Aid for Young Scientists Research (B) 
No.16K17619 
and 
by Grant-in-Aid for Young Scientists Research (B) 
No.15K17571.


\end{document}